\documentclass[11pt]{article}

\usepackage{amsmath,amssymb,amsthm,accents}
\usepackage{amscd}
\usepackage[initials,short-journals,non-sorted-cites]{amsrefs}
\usepackage{enumitem}
\usepackage{type1cm}
\usepackage{float}
\usepackage[all]{xy}
\usepackage{amsfonts}
\usepackage[dvipdfmx]{graphicx}
\usepackage{comment}
\usepackage{bm}

\usepackage{color}

\allowdisplaybreaks[4]

\renewcommand{\abstract}[1]{
	\begin{center}
	\parbox{13cm}{\small {\sc Abstract.} #1}
	\end{center}
	\smallskip
}

\newcommand{\address}[1]{\bigskip\small\noindent #1 \par}
\newcommand{\email}[1]{\small\noindent\textit{Email address}: \texttt{#1} \par}

\newtheorem{theo}{Theorem}[section]
\newtheorem{prop}[theo]{Proposition}
\newtheorem{lemm}[theo]{Lemma}

\newtheorem{Theo}{Theorem}

\theoremstyle{definition}

\newtheorem{exam}[theo]{Example}

\newtheorem{remark}[theo]{Remark}

\newcommand{\Z}{\mathbb{Z}}

\title{A homogeneous presentation of symmetric quandles\footnotemark}
\author{Yuta Taniguchi}
\date{}

\begin{document}

\maketitle
\footnotetext[\ast]{
{\bf keywords:} {quandle, symmetric quandle}
}

\abstract{A quandle is an algebraic structure whose axioms correspond to the Reidemeister moves of knot theory. S. Kamada introduced the notion of a quandle with a good involution, which was later called a symmetric quandle. We are interested in the algebraic structure of symmetric quandles. Given a group $G$, an element $z$ and a certain subgroup $H$, one can obtain the quandle. D. Joyce showed that every quandle is isomorphic to the disjoint union of such quandles. In this paper, given a group $G$, elements $z,r$ in $G$ and a certain subgroup $H$, we construct a symmetric quandle. Furthermore, we show that every symmetric quandle is isomorphic to the disjoint union of such quandles.
}

\section{Introduction}
A quandle \cite{Joy,Mat} is an algebraic structure whose axioms correspond to the Reidemeister moves of knot theory. Given an oriented $n$-dimensional link $L$, we have the link quandle of $L$ \cite{Joy,Mat}, which is a generalization of the link group $\pi_1(S^{n+2}\backslash L)$. J. S. Carter, D. Jelsovsky, S. Kamada, L. Langford and M. Saito introduced the homology group of quandles and defined invariants of $1$- or $2$-dimensional oriented links, which are called the quandle cocycle invariants \cite{CJKLS}.  We need the orientation of links to define the quandle cocycle invariants.

S. Kamada introduced the notion of a quandle with a good involution \cite{Kamada2007quan}, which was later called a symmetric quandle \cite{Kamada2010homo}, and associated a symmetric quandle to an $n$-dimensional unoriented (possibily nonorientable) link. In \cite{Kamada2010homo}, the quandle cocycle invariants of unoriented (possibily nonorientable) $2$-dimensional links were defined using homology groups of symmetric quandles. 

We are interested in the algebraic structure of symmetric quandles. Given a group $G$, an element $z$ and a certain subgroup $H$, one can obtain the quandle $(H\backslash G;z)$. D. Joyce showed that every quandle is isomorphic to the disjoint union of such quandles $(H\backslash G;z)$ \cite{Joy}. In this paper, given a group $G$, elements $z,r$ in $G$ and a certain subgroup $H$, we construct a symmetric quandle $(H\backslash G;z,r)$. Furthermore, we show that every symmetric quandle is isomorphic to the disjoint union of such quandles $(H\backslash G;z,r)$.
\begin{Theo}[Theorem~\ref{theo:main2}]
Every symmetric quandle is representable as $(\sqcup_{i\in I}H_i\backslash G;z_i,r_i\mid i\in I)$.
\end{Theo}

 This paper is organized as follows. In section \ref{sec:def}, we review the definitions of a quandle and a good involution. In section \ref{sec:main}, we prove the main results.
\section*{Acknowledgements}
The author would like to thank Seiichi Kamada and Jumpei Yasuda for helpful advice and discussions on this research. The author would also like to thank the editor and the referee for many valuable comments and suggestions. This work was supported by JSPS KAKENHI Grant Number 21J21482.
\section{Symmetric quandle}
\label{sec:def}
A {\it quandle} \cite{Joy,Mat} is a set $Q$ with a binary operation $\ast:Q\times Q\to Q$ satisfying the following three axioms.
\vspace{-0.5\baselineskip}
\begin{itemize}
	\setlength{\itemsep}{0pt}
	\setlength{\parskip}{0pt}
\item[(Q$1$)] For any $a\in Q$, we have $a\ast a=a$.
\item[(Q$2$)] For any $b\in Q$, the map $s_b:Q\to Q$, $a\mapsto a\ast b$ is a bijection.
\item[(Q$3$)] For any $a,b,c\in Q$, we have $(a\ast b)\ast c=(a\ast c)\ast(b\ast c)$.
\end{itemize}

\vspace{-0.4\baselineskip}
These axioms correspond to Reidemeister moves in knot theory. By the second axiom, there exists the binary operation $\bar{\ast}:Q\times Q\to Q$ such that $(a\ast b)\bar{\ast}b=(a\bar{\ast}b)\ast b=a$ for any elements $a,b\in Q$. We call $\bar{\ast}$ the {\it dual operation} of $(Q,\ast)$.

Let $Q$ be a quandle. A map $\rho:Q\to Q$ is a {\it good involution} \cite{Kamada2007quan} if $\rho$ is an involution such that $\rho(a\ast b)=\rho(a)\ast b$ and $a\ast\rho(b)=a\bar{\ast}b$ for any $a,b\in Q$. Such a pair $(Q,\rho)$ is called a {\it symmetric quandle}. 

\begin{exam}
A quandle $Q$ is a {\it kei} or an {\it involutory quandle} if $(a\ast b)\ast b=a$ for any $a,b\in Q$. Then, the identity map on $Q$ is a good involution.
\end{exam}
\begin{exam}
Let $G$ be a group. We define an operation $\ast$ on $G$ by $a\ast b:=b^{-1}ab$. Then, ${\rm Conj}(G)=(G,\ast)$ is a quandle, which is called the {\it conjugation quandle} of $G$. The inversion ${\rm Inv}(G):G\to G;a\mapsto a^{-1}$ is a good involution. We call $({\rm Conj}(G),{\rm Inv}(G))$ the {\it conjugation symmetric quandle} of $G$. 
\end{exam}

Let $(Q,\rho)$ and $(Q^{\prime},\rho^{\prime})$ be symmetric quandles. A map $f:Q\to Q^{\prime}$ is a {\it quandle homomorphism} if $f(a\ast b)=f(a)\ast f(b)$ for any $a,b\in Q$. A quandle homomorphism $f:Q\to Q^{\prime}$ is a {\it quandle isomorphism} if it is a bijection. A quandle homomorphism $f:Q\to Q^{\prime}$ is a {\it symmetric quandle homomorphism} if $f\circ\rho=\rho^{\prime}\circ f$. A symmetric quandle homomorphism $f:Q\to Q^{\prime}$ is a {\it symmetric quandle isomorphism} or a {\it symmetric quandle automorphism} if it is a bijection, or if $(Q,\rho)=(Q^{\prime},\rho^{\prime})$ and it is a bijection, respectively. We denote the set of all symmetric quandle automorphisms of $(Q,\rho)$ by ${\rm Aut}(Q,\rho)$. Then, ${\rm Aut}(Q,\rho)$ forms a group by $f\cdot g(a)=g\circ f(a)$ for $f,g\in{\rm Aut}(Q,\rho)$ and $a\in Q$. Notice that $s_a$ is a symmetric quandle automorphism for any $a\in Q$. 

The group ${\rm Aut}(Q,\rho)$ acts on $Q$ from the right by $a\cdot f=f(a)$ for $a\in Q$ and $f\in {\rm Aut}(Q,\rho)$. A symmetric quandle $(Q,\rho)$ is {\it homogeneous} if the action of ${\rm Aut}(Q,\rho)$ on $Q$ is transitive.

\begin{lemm}
\label{lemm:symmetric qdle}
Let $(Q,\rho)$ be a symmetric quandle.\\
$(1)$ For any $a\in Q$ and $f\in{\rm Aut}(Q,\rho)$, we have $s_{a\cdot f}=f^{-1}s_af$.\\
$(2)$ For any $a\in Q$, we have $s_{\rho(a)}=s_a^{-1}$.
\end{lemm}
\begin{proof}
This follows from the definitions.
\end{proof}

\section{Main results}
\label{sec:main}
Let $G$ be a group, $H$ be a subgroup of $G$ and $z$ be an element of $G$. Assume that $H$ satisfies the following condition:
\begin{itemize}
\item For any $h\in H$, we have $z^{-1}hz=h$.
\end{itemize}
We define the binary operation $\ast$ on the set of right cosets $H\backslash G$ by $Hx\ast Hy:=Hxy^{-1}zy$ for $Hx,Hy\in H\backslash G$. Then, the pair $(H\backslash G,\ast)$ satisfies the conditions (Q$2$) and (Q$3$). Hence, the pair $(H\backslash G,\ast)$ is a {\it rack} \cite{fenn1992racks}. We note that the dual operation $\bar{\ast}$ is given by $Hx\bar{\ast}Hy:=Hxy^{-1}z^{-1}y$ for $Hx,Hy\in H\backslash G$.

Suppose that $H$ satisfies the following condition:
\begin{itemize}
\item We have $z\in H$.
\end{itemize}
Then, the pair $(H\backslash G,\ast)$ satisfies the condition (Q$1$), which implies that the pair $(H\backslash G,\ast)$ is a quandle (cf. \cite{Joy}), which is denoted by $(H\backslash G;z)$.

Furthermore, we assume that there exists an element $r\in G$ which satisfies the following conditions: 
\begin{itemize}
\item For any $h\in H$, we have $rhr^{-1}\in H$.
\item We have $r^2\in H$.
\item We have $r^{-1}zr=z^{-1}$
\end{itemize}

\begin{prop}
Let us define $\rho_r$ as follows: $\rho_r(Hx):=Hrx \ (Hx\in H\backslash G)$. Then, $\rho_r:H\backslash G\to H\backslash G$ is well defined and $\rho_r$ is a good involution.
\end{prop}
\begin{proof}
If $Hx_1=Hx_2$, there is an element $h\in H$ such that $x_1=hx_2$. Since $rhr^{-1}\in H$, we have $Hrx_1=Hrhx_2=Hrhr^{-1}rx_2=Hrx_2$. Thus, $\rho_r$ is well defined.

Since $r^2\in H$, we have $\rho_r^2(Hx)=\rho_r(Hrx)=Hr^2x=Hx$ for any $Hx\in H\backslash G$, which implies that $\rho_r$ is an involution. For any $Hx,Hy\in H\backslash G$, it holds that 
 \begin{align*}
 \rho_r(Hx\ast Hy)&=\rho_r(Hxy^{-1}zy)=Hrxy^{-1}zy\\
 &=Hrx\ast Hy=\rho_r(Hx)\ast Hy
 \end{align*}
 and 
 \begin{align*}
 Hx\ast\rho_r(Hy)& = Hx\ast Hry =Hx(ry)^{-1}zry\\
 &=Hxy^{-1}r^{-1}zry=Hxy^{-1}z^{-1}y\\
 &=Hx\bar{\ast} Hy.
 \end{align*}
 Hence, $\rho_r$ is a good involution.
\end{proof}
Let us denote this symmetric quandle $((H\backslash G;z),\rho_r)$ by $(H\backslash G;z,r)$.

\begin{theo}
\label{theo:main1}
 Every homogeneous symmetric quandle is representable as $(H\backslash G;z,r)$. 
\end{theo}
\begin{proof}
Let $(Q,\rho)$ be a homogeneous symmetric quandle, $q$ be an element of $Q$ and $G$ be the symmetric automorphism group ${\rm Aut}(Q,\rho)$. Let us define the subgroup $H\subset G$ by $\{ x\in G\mid q\cdot x=q\}$. 

Since $(Q,\rho)$ is homogeneous, the action of $G$ is transitive. Thus, there exists an element $g\in G$ such that $q\cdot g=\rho(q)$. We will see that $G,H,z:=s_q$ and $r:=g$ satisfy the conditions as stated above. 

Let $h$ be an element of $H$. Since $q\cdot h=q$, we have $h^{-1}s_q h=s_{q\cdot h}=s_q$, which implies that $s_{q}^{-1}hs_q=h$. The condition $s_q\in H$ is clear. By the definition of symmetric quandle isomorphisms, it holds that $\rho(a)\cdot f=\rho(a\cdot f)$ for any $a\in Q$ and $f\in {\rm Aut}(Q,\rho)$. Thus, we have 
\begin{align*}
q\cdot(ghg^{-1})&=\rho(q)\cdot(hg^{-1})=\rho(q\cdot h)\cdot g^{-1}\\
&=\rho(q)\cdot g^{-1}=q,
\end{align*}
which implies that $ghg^{-1}\in H$. Since $q\cdot(g^2)=\rho^2(q)=q$, we have $g^2\in H$. By lemma~\ref{lemm:symmetric qdle}, we have $g^{-1}s_{q}g=s_{q\cdot g}=s_{\rho(q)}=s_{q}^{-1}$. Hence, $G,H,z$ and $r$ satisfy the sufficient conditions that precede proposition 3.1. This implies that the symmetric quandle $(H\backslash G;z,r)$ can be defined. We show that $(Q,\rho)$ is symmetric quandle isomorphic to $(H\backslash G;z,r)$.

It is known that $\psi:H\backslash G\to Q: Hx\mapsto q\cdot x$ is a quandle isomorphism \cite{Joy}. For any $Hx\in H\backslash G$, we have
\begin{align*}
\psi\circ\rho_{r}(Hx)&=\psi(Hrx) =q\cdot(rx)\\
&=\rho(q)\cdot x=\rho(q\cdot x)\\
&=\rho\circ\psi(Hx).
\end{align*}
Hence, $\psi$ is a symmetric quandle isomorphism.
\end{proof}
Let $G$ be a group, $I$ be an index set, $H_i (i\in I)$ be subgroups of $G$ and $z_i(i\in I)$ be elements of $G$. Assume that $H_i$ satisfies the following condition for each $i\in I$:
\begin{itemize}
\item For any $h_i\in H_i$, we have $z_i^{-1}h_iz_i=h_i$.
\end{itemize}
We define the binary operation $\ast$ on the disjoint union of the set of right cosets $\sqcup_{i\in I}H_i\backslash G$ by $H_ix\ast H_jy:=H_ixy^{-1}z_jy$ for $H_ix,H_jy\in\sqcup_{i\in I}H_i\backslash G$. Then, the pair $(\sqcup_{i\in I}H_i\backslash G,\ast)$ satisfies the conditions (Q$2$) and (Q$3$). Hence, the pair $(\sqcup_{i\in I}H_i\backslash G,\ast)$ is a rack. Notice that the dual operation $\bar{\ast}$ is given by $H_ix\bar{\ast} H_jy:=H_ixy^{-1}z_j^{-1}y$ for $H_ix,H_jy\in\sqcup_{i\in I}H_i\backslash G$.

Suppose that $H_i$ satisfies the following condition for each $i\in I$:
\begin{itemize}
\item We have $z_i\in H_i$.
\end{itemize}
Then, the pair $(\sqcup_{i\in I}H_i\backslash G,\ast)$ satisfies the condition (Q$1$), which implies that the pair $(\sqcup_{i\in I}H_i\backslash G,\ast)$ is a quandle (cf. \cite{Joy}), which is denoted by $(\sqcup_{i\in I}H_i\backslash G;z_i\mid i\in I)$.

 Furthermore, we assume that there exist an involution $\kappa:I\to I$ and elements $r_i\in G (i\in I)$ which satisfy the following conditions:
\begin{itemize}
\item For any $i\in I$ and $h_i\in H_i$, we have $r_ih_ir_i^{-1}\in H_{\kappa(i)}$.
\item For any $i\in I$, we have $r_{\kappa(i)}r_i\in H_i$.
\item For any $i\in I$, we have $z_i^{-1}=r_i^{-1}z_{\kappa(i)}r_i$.
\end{itemize}

\begin{prop}
 We put ${\bm r}:=(r_i\mid i\in I)$. Let us define $\rho_{\bm r}$ as follows: $\rho_{\bm r}(H_ix):=H_{\kappa(i)}r_ix\ (H_ix\in\sqcup_{i\in I}H_i\backslash G)$. Then, $\rho_{\bm r}:\sqcup_{i\in I}H_i\backslash G\to\sqcup_{i\in I}H_i\backslash G$ is well defined and $\rho_{\bm r}$ is a good involution.
\end{prop}
\begin{proof}
 If $H_ix_1=H_ix_2$, there exists an element $h_i\in H_i$ such that $x_1=h_ix_2$. Since $r_ih_ir_i^{-1}\in H_{\kappa(i)}$, we see that $H_{\kappa(i)}r_ix_1=H_{\kappa(i)}r_ih_ix_2=H_{\kappa(i)}r_ih_ir_i^{-1}r_ix_2=H_{\kappa(i)}r_ix_2$. Hence, $\rho_{\bm r}$ is well defined.

Since $r_{\kappa(i)}r_i\in H_i$ for any $i\in I$, it holds that $\rho_{\bm r}^2(H_ix)=\rho_{\bm r}(H_{\kappa(i)}r_ix)=H_ir_{\kappa(i)}r_ix=H_ix$ for any $H_ix\in\sqcup_{i\in I}H_i\backslash G$. Thus, $\rho_{\bm r}$ is an involution. 

For any $H_ix,H_jy\in\sqcup_{i\in I}H_i\backslash G$, we have
\begin{align*}
\rho_{\bm r}(H_ix\ast H_jy)&=\rho_{\bm r}(H_ixy^{-1}z_jy)=H_{\kappa(i)}r_ixy^{-1}z_jy\\
&=H_{\kappa(i)}r_ix\ast H_jy=\rho_{\bm r}(H_ix)\ast H_jy
\end{align*}
and 
\begin{align*}
H_ix\ast\rho_{\bm r}(H_jy)&=H_ix\ast H_{\kappa(j)}r_jy=H_ix(r_jy)^{-1}z_{\kappa(j)}r_jy\\
&=H_ixy^{-1}r_j^{-1}z_{\kappa(j)}r_jy=H_ixy^{-1}z_j^{-1}y\\
&=H_ix\bar{\ast}H_jy.
\end{align*}
Hence, $\rho_{\bm r}$ is a good involution.
\end{proof}
Let us denote this symmetric quandle $((\sqcup_{i\in I}H_i\backslash G;z_i\mid i\in I),\rho_{\bm r})$ by $(\sqcup_{i\in I}H_i\backslash G;z_i,r_i\mid i\in I)$.
\begin{theo}
\label{theo:main2}
Every symmetric quandle is representable as $(\sqcup_{i\in I}H_i\backslash G;z_i,r_i\mid i\in I)$.
\end{theo}
\begin{proof}
Let $(Q,\rho)$ be a symmetric quandle and $G$ be the symmetric automorphism group ${\rm Aut}(Q,\rho)$. Let $Q=\sqcup_{i\in I}Q_i$ be the orbit decomposition by the action of $G$ on $Q$. For each $i\in I$, we fix an element $q_i\in Q_i$. Let $H_i$ be the subgroup of $G$ defined by $\{x\in G\mid q_i\cdot x=q_i\}$ for each $i\in I$.

Let $\kappa:I\to I$ be the involution which satisfies $\rho(q_i)\in Q_{\kappa(i)}$ for any $i\in I$. Since $G$ acts on $Q_i$ transitively, there exists $g_i\in G$ such that $\rho(q_i)=q_{\kappa(i)}\cdot g_i$ for each $i\in I$. We will see the $G,H_i(i\in I),z_i:=s_{q_i}(i\in I),\kappa$ and $r_i:=g_i (i\in I)$ satisfy the conditions as stated above. 

Let $i$ be an element of $I$ and $h_i$ be an element of $H_i$. Since $q_i\cdot h_i=q_i$, it holds that $h_i^{-1}s_{q_i}h_i=s_{q_i\cdot h_i}=s_{q_i}$, which implies that $s_{q_i}^{-1}h_is_{q_i}=h_i$. The condition $s_{q_i}\in H_i$ is clear. By the definition of symmetric quandle isomorphisms, we have
\begin{align*}
q_{\kappa(i)}\cdot(g_ih_ig_i^{-1})&=\rho(q_i)\cdot(h_ig_i^{-1})=\rho(q_i\cdot h_i)\cdot g_i^{-1}\\
&=\rho(q_i)\cdot g_i^{-1}=q_{\kappa(i)},
\end{align*} 
which implies that $g_ih_ig_i^{-1}\in H_{\kappa(i)}$. Since $q_{i}\cdot g_{\kappa(i)}=q_{\kappa^2(i)}\cdot g_{\kappa(i)}=\rho(q_{\kappa(i)})$, we have 
\begin{align*}
q_{i}\cdot (g_{\kappa(i)}g_i)&=\rho(q_{\kappa(i)})\cdot g_i=\rho(q_{\kappa(i)}\cdot g_i)\\
&=\rho^2(q_i)=q_i.
\end{align*}
By lemma~\ref{lemm:symmetric qdle}, $s_{q_i}^{-1}=s_{\rho(q_i)}=s_{q_{\kappa(i)}\cdot g_i}=g_i^{-1}s_{q_{\kappa(i)}}g_i$.

Hence, the symmetric quandle $(\sqcup_{i\in I}H_i\backslash G;z_i,r_i\mid i\in I)$ can be defined. We show that $(\sqcup_{i\in I}H_i\backslash G;z_i,r_i\mid i\in I)$ is symmetric quandle isomorphic to $(Q,\rho)$.

It is known that $\psi:\sqcup_{i\in I}H_i\backslash G\to Q:H_ix\mapsto q_i\cdot x$ is a quandle isomorphism \cite{Joy}. For $H_ix\in \sqcup_{i\in I}H_i\backslash G$, we have 
\begin{align*}
\psi\circ\rho_{\bm r}(H_ix)&=\psi(H_{\kappa(i)}r_ix)=q_{\kappa(i)}\cdot(r_ix)\\
&=\rho(q_i)\cdot x=\rho(q_i\cdot x)\\
&=\rho\circ\psi(H_ix).
\end{align*}
Thus, $\psi$ is a symmetric quandle isomorphism.
\end{proof}
\begin{remark}
Let $(Q,\rho)$ be a symmetric quandle. The subgroup of ${\rm Aut}(Q,\rho)$ generated by $\{ s_a\mid a\in Q\}$ is called the {\it inner automorphism group} of $(Q,\rho)$ and denoted by ${\rm Inn}(Q,\rho)$. If we replace ${\rm Aut}(Q,\rho)$ to ${\rm Inn}(Q,\rho)$, the argument holds as in Theorem~\ref{theo:main2}.
\end{remark}
\begin{exam}
{\rm
Let $G$ be a group. Suppose that $G$ acts on $G$ from the left by $g\cdot x=x^{-1}gx$. Let $G=\sqcup_{i\in I}G_i$ be the orbit decomposition by the action of $G$. For each $i$, we fix an element $g_i\in G_i$. We assume that the set $\{ g_i\mid i\in I\}$ is closed under the inversion, that is, ${\rm Inv}(G)(\{ g_i\mid i\in I\})=\{ g_i\mid i\in I\}$. Let us define the subgroup $H_i$ by $\{ x\in G\mid xg_i=g_ix\}$ for each $i\in I$. 

Let $\kappa:I\to I$ be the involution which satisfies $g_{\kappa(i)}={\rm Inv}(G)(g_i)=g_i^{-1}$ for any $i\in I$. We put $r_i=e$ for each $i\in I$. Then, the symmetric quandle $(\sqcup_{i\in I}H_i\backslash G;z_i,r_i\mid i\in I)$ can be defined. We define $\psi$ as follows: $\psi(H_ix):=g_i\cdot x=x^{-1}g_ix\ (H_ix\in \sqcup_{i\in I}H_i\backslash G)$. Since $h_i^{-1}g_ih_i=g_i$ for any $h_i\in H_i$, $\psi:\sqcup_{i\in I}H_i\backslash G\to G$ is well defined. By the definition of $H_i$ and $\psi$, it holds that $\psi$ is a bijection. For each $H_ix,H_jy\in \sqcup_{i\in I}H_i\backslash G$, we have
\begin{eqnarray*}
\psi(H_ix\ast H_iy)&=&\psi(H_ixy^{-1}g_jy)\\
&=&g_i\cdot(xy^{-1}g_jy)\\
&=&(xy^{-1}g_jy)^{-1}g_i(xy^{-1}g_jy)\\
&=&(y^{-1}g_jy)^{-1}(x^{-1}g_ix)(y^{-1}g_jy)\\
&=&(g_j\cdot y)^{-1}(g_i\cdot x)(g_j\cdot y)=\psi(H_jy)^{-1}\psi(H_ix)\psi(H_jy).
\end{eqnarray*}
Thus, we see that $\psi$ is a quandle homomorphism from $(\sqcup_{i\in I}H_i\backslash G;z_i\mid i\in I)$ to ${\rm Conj}(G)$. Furthermore, it holds that
\begin{eqnarray*}
\psi\circ\rho_{\bm r}(H_ix)&=&\psi(H_{\kappa(i)}x)\\
&=&g_{\kappa(i)}\cdot x\\
&=&g_i^{-1}\cdot x\\
&=&x^{-1}g_i^{-1}x\\
&=&(g_i\cdot x)^{-1}={\rm Inv}(G)(g_i\cdot x)={\rm Inv}(G)\circ\psi(H_ix)
\end{eqnarray*}
for any $H_ix\in \sqcup_{i\in I}H_i\backslash G$. Hence, the symmetric quandle $(\sqcup_{i\in I}H_i\backslash G;z_i,r_i\mid i\in I)$ is symmetric quandle isomorphic to the symmetric conjugation quandle $({\rm Conj}(G),{\rm Inv}(G))$.
}
\end{exam}
\begin{exam}
{\rm
Let $\Z_4$ be the cyclic group of order $4$. Let us define an operation $\ast$ on $\Z_4$ by $x\ast y=2y-x$. Then, this pair $R_4=(\Z_4,\ast)$ is a quandle, which is called the {\it dihedral quandle} of order $4$. We define an involution $\rho_{\rm anti}:R_4\to R_4$ by $\rho_{\rm anti}(x)=x+2$. Then, $\rho_{\rm anti}$ is a good involution, which is called the {\it antipodal map}~\cite{Kamada2010homo}.

Let $G$ be the quaternion group. We note that $\langle a,b,c\mid a^2=b^2=c^2=abc, a^4=e\rangle$ is a presentation of $G$, where $e$ is the identity element of $G$. Thus, it holds that $ab=c,bc=a$ and $ca=b$ in $G$.

 Let $H_1$ be the subgroup of $G$ generated by $a$ and $H_2$ be the subgroup of $G$ generated by $b$. Then, it holds that $|G|=8$ and $|H_1|=|H_2|=4$. By direct calculations, we have $H_1\backslash G=\{ H_1e,H_1b\}$ and $H_2\backslash G=\{ H_2e,H_2a\}$. We put $I=\{1,2\}$, $z_1=a$ and $z_2=b$. Since $H_i$ is generated by $z_i$, the quandle $(\sqcup_{i\in I}H_i\backslash G;z_i\mid i\in I)$ can be defined. Then, we have 
\begin{eqnarray*}
&&H_1e\ast H_1b=H_1b^{-1}ab=H_1b^{-1}c=H_1a^{-1}=H_1e,\\
&&H_1e\ast H_2e=H_1b,\\
&&H_2e\ast H_2a=H_2a^{-1}ba=H_2c^{-1}a=H_2b^{-1}=H_2e\ \rm{and}\\
&&H_2e\ast H_1e=H_2a.
\end{eqnarray*}
Hence, a map $\psi:H_1\backslash G\sqcup H_2\backslash G\to R_4$ defined by $\psi(H_1e)=0,\psi(H_1a)=2,\psi(H_2e)=1$ and $\psi(H_2b)=3$ is a quandle isomorphism.

Let $\kappa:I\to I$ be the identity map. We put $r_1=b$ and $r_2=a$. Since $aba^{-1}=b^{-1}$ and $bab^{-1}=a^{-1}$, we can define the symmetric quandle $(\sqcup_{i\in I}H_i\backslash G;z_i,r_i\mid i\in I)$. By the definition of $\rho_{\bm r}$, we have $\rho_{\bm r}(H_1e)=H_1b$ and $\rho_{\bm r}(H_2e)=H_2a$. Hence, $\psi$ is a symmetric quandle isomorphism from $(\sqcup_{i\in I}H_i\backslash G;z_i,r_i\mid i\in I)$ to $(R_4,\rho_{\rm anti})$.
}
\end{exam}

\bibliographystyle{plain}
\bibliography{reference}

\address{{\scriptsize DEPARTMENT OF MATHEMATICS, GRADUATE SCHOOL OF SCIENCE, OSAKA UNIVERSITY, 1-1, MACHIKANEYAMA, TOYONAKA, OSAKA, 560-0043, JAPAN}}

\email{u660451k@ecs.osaka-u.ac.jp} 
\end{document}